\documentclass{amsart}
\allowdisplaybreaks

\usepackage{amsmath}
\usepackage{amsfonts}
\usepackage{amssymb}
\usepackage{graphicx}

\usepackage{soul}
\usepackage{todonotes}

\usepackage{amsthm}
\usepackage{graphicx}

\usepackage{mathrsfs}
\usepackage[all]{xy}
\usepackage{hyperref}

\DeclareMathOperator{\Spec}{Spec}

\newcommand{\FF}{\mathbb{F}}
\DeclareMathOperator{\PSL}{PSL}
\DeclareMathOperator{\Pic}{Pic}

\DeclareMathOperator{\GL}{GL}
\DeclareMathOperator{\Sp}{Sp}
\DeclareMathOperator{\diag}{diag}

\DeclareMathOperator{\Aut}{Aut}
\DeclareMathOperator{\Sing}{Sing}
\newcommand{\sh}{\mathscr}
\newcommand{\ZZ}{\mathbb{Z}}
\newcommand{\CC}{\mathbb{C}}
\newcommand{\PP}{\mathbb{P}}
\newcommand{\M}{\mathcal{M}}

\newcommand{\J}{\mathcal{J}}
\newcommand{\R}{\mathcal{R}}

\newcommand{\KM}[1]{V_4\mathcal{M}_{#1}}
\newcommand{\ZM}[1]{\mathbb{Z}_2\mathcal{M}_{#1}}
\newcommand{\ZZM}[1]{\mathbb{Z}_2^2\mathcal{M}_{#1}}

\renewcommand{\S}{\mathcal{S}}

\newtheorem*{rep@theorem}{\rep@title}
\newcommand{\newreptheorem}[2]{%
\newenvironment{rep#1}[1]{%
 \def\rep@title{#2 \ref{##1}}%
 \begin{rep@theorem}}%
 {\end{rep@theorem}}}
\makeatother

\newreptheorem{theorem}{Theorem}

\newtheorem{theorem}{Theorem}[section]
\theoremstyle{plain}
\newtheorem{corollary}[theorem]{Corollary}
\newtheorem{definition}[theorem]{Definition}

\newtheorem{lemma}[theorem]{Lemma}
\newtheorem{proposition}[theorem]{Proposition}
\newtheorem{remark}[theorem]{Remark}

\begin{document}

\title{The Moduli of Klein Covers of Curves}
\author{Charles Siegel}
\address[Charles Siegel]{Kavli Institute for the Physics and Mathematics of the Universe (WPI), Todai Institutes for Advanced Study, the University of Tokyo}
\email[Charles Siegel]{charles.siegel@ipmu.jp}%
\urladdr{http://db.ipmu.jp/member/personal/2754en.html}
\date{}

\begin{abstract}
We study the moduli space $\KM{g}$ of Klein four covers of genus $g$ curves and its natural compactification.  This requires the construction of a related space which has a choice of basis for the Klein four group.  This space has the interesting property that the two components intersect along a component of the boundary.  Further, we carry out a detailed analysis of the boundary, determining components, degrees of the components over their images in $\overline{\M_g}$, and computing the canonical divisor of $\overline{\KM{g}}$.
\end{abstract}

\maketitle

\tableofcontents

\section*{Introduction}

Ever since the nineteenth century, unramified double covers have been an essential tool for studying curves.  They correspond to square roots of the trivial line bundle which form a group.  This formulation can be used to study surface groups and the mapping class group, and also theta characteristics, the square roots of the canonical line bundle.  There are intricate relationships between double covers and theta characteristics, in particular, and the difference between them only become completely clear after passing to moduli.

The moduli of double covers of curves has two components, one that is isomorphic to $\M_g$, where the double cover is a disconnected union of two copies of the base curve, and one where the double cover is nontrivial, traditionally denoted by $\R_g$.  While the moduli of theta characteristics also has two components, neither one is isomorphic to $\M_g$.  The components, $\S_g^\pm$, correspond to whether the theta characteristic has an even or odd dimensional space of global sections.\footnote{The notations come from the French for covering, rev\^etement and from the fact that curves with theta characteristics are often called spin curves, due to connections with the quantum mechanical notion of spin.}

Another, slightly more subtle, connection between the two moduli spaces is that the theta characteristics on a curve correspond to quadratic forms on the ($\FF_2$-vector space of) points of order two on the curve.  The quadratic form is given by, if $L$ is a theta characteristic, $\mu\mapsto h^0(L\otimes \mu)-h^0(L)\mod 2$, and induces a skew-symmetric bilinear form on the points of order two.  This bilinear form is independent of the theta characteristic chosen and is called the Weil pairing.  The Weil pairing, however, is really an invariant of a Klein four subgroup of the Jacobian, as $\langle \mu,\nu\rangle=\langle\mu,\mu+\nu\rangle$, and in fact, if $\{0,\mu_1,\mu_2,\mu_3\}$ is a Klein four subgroup of $\J(C)[2]$, then the Weil pairing on the group has the value $h^0(L\otimes\sh{O}_C)+h^0(L\otimes \mu_1)+h^0(L\otimes \mu_2)+h^0(L\otimes \mu_3)\mod 2$, which is manifestly symmetric in the group elements, suggesting that it could be clarified by studying the moduli of Klein covers.

The boundaries of these moduli spaces have been studied in detail, and the fibers of the natural projection to $\overline{\M_g}$ have been made very explicit.  The approach originates in {\cite{MR1082361}} for $\overline{\S_g}^{\pm}$ and is pushed through in detail in {\cite{MR2007379}}, and adapted to $\overline{\R_g}$ in {\cite{MR2117416}}.  More recently, this approach has been adapted to proving that pluricanonical forms extend to both $\overline{\S_g}^{\pm}$ {\cite{MR2551759}} and to $\overline{\R_g}$ {\cite{MR2639318}}, bringing the study of the birational geometry of these spaces into reach.

\subsection*{This paper}

In this paper, we extend the description of the boundary and pluricanonical forms to the moduli of Klein covers of curves.  This, however, is difficult to do directly, so instead we introduce an intermediate moduli space, $\overline{\ZZM{g}}$, of pairs of Prym curves and study it, then use the relationship between it and the moduli of Klein covers $\overline{\KM{g}}$ to prove the results on this space.  In fact, we use this relationship to define $\overline{\KM{g}}$.

In section 1 of this paper, we recall relevant facts about double covers. In particular, the classification of points in the fiber over a stable curve from {\cite{MR2117416}}, and the relationship between two competing notations for the components of the boundary of $\overline{\R_g}$, used in, for instance, {\cite{MR903385}} and \cite{MR2976944}.

In section 2, we initiate the study of $\overline{\ZZM{g}}$, focusing on the interior.  We construct the space and show that there are two connected components, corresponding to Weil pairing 0, which was studied in \cite{1302.5946} under the notation $\mathcal{R}^2\mathcal{M}_g$, and Weil pairing 1, and the degree of each component over $\M_g$, reproducing a result in {\cite{1206.5498}}, which holds in the degenerate case where the dihedral group is only four elements.

In section 3, we analyze the boundary of $\overline{\ZZM{g}}$ in detail, describing the fibers over $\overline{\M_g}$, then identify the boundary components and determine how many objects in each fiber lie in each component.  Here, we note an interesting fact.  Although $\ZZM{g}$ is an unramified covering of $\M_g$ and has two components, the natural compactification $\overline{\ZZM{g}}$ is in fact connected, and the boundaries of the two components intersect nicely along a single component.

In section 4, we proceed to analyzing $\overline{\KM{g}}$ and its boundary.  We do so by showing that the group action of $\PSL_2(\FF_2)$ on $\ZZM{g}$ extends to the boundary of each of the two components separately, identifying several components of $\partial\overline{\ZZM{g}}$.  This allows us to describe the (much simple) boundary of $\overline{\KM{g}}$and to show that the natural map $\overline{\KM{g}}\to \overline{\M_g}$ is simply ramified along a single boundary component.

In the last section, we follow \cite{MR2639318}, \cite{MR2551759} and \cite{MR664324}, to extend the pluricanonical forms from the smooth locus to an arbitrary resolution of singularities.  The main tool in this is the Reid--Shepherd-Barron--Tai criterion \cite{MR605348,MR763023}.  We conclude with a slope criterion for $\overline{\KM{g}}^i$ to be of general type analagous to similar results for $\overline{\M}_g$ and $\overline{\R_g}$:

\begin{reptheorem}{maintheorem}
For any $g$, $\overline{\KM{g}}^i$ has general type if there exists a single effective divisor $D\equiv a\lambda-\sum_T b_{\Delta_T}\Delta_T$ where $T$ runs over all boundary components, such that all the ratios $\frac{a}{b_T}$ are less than $\frac{13}{2}$ and the ratios $\frac{a}{b_{II,III,III}}$, $\frac{a}{b_{1,g-1,1:g-1}}$, $\frac{a}{b_{1,1,1}}$, $\frac{a}{b_{g-1,g-1,g-1}}$, $\frac{a}{b_{1,1:g-1,1:g-1}}$, $\frac{a}{b_{g-1,1:g-1,1:g-1}}$, and $\frac{a}{b_{1:g-1,1:g-1,1:g-1}}$ are less than $\frac{13}{3}$.
\end{reptheorem}

\subsection*{Acknowledgments}
I would like to thank Gavril Farkas, for suggesting that the birational geometry of this space might be interesting, as well as for conversations on the relationship between the Weil pairing, theta characteristics and Klein four curves.  Also I would like to thank Angela Gibney, Joe Harris, Tyler Kelly and Angela Ortega for helpful discussions on the moduli of curves, coverings, the Weil pairing and birational geometry and Amir Aazami and Jesse Wolfson for comments on an earlier draft.  This work was supported by World Premier International Research Center Initiative (WPI Initiative), MEXT, Japan.

\section{Background}

In this section, we will recall relevant facts about double covers and points of order two on curves.  For $C$ a smooth projective curve over $\CC$, we denote by $\J(C)$ the group of line bundles of degree $0$ on $C$.  It has a natural subgroup $\J(C)[2]$ consisting of the elements whose square is trivial.

\begin{lemma}
\label{lemma:prymdefs}
The following data are equivalent:

\begin{enumerate}
 \item $\tilde{C}\to C$ an irreducible \'etale double cover,
 \item $\mu\in \J(C)[2]$ nonzero, and
 \item $\tilde{C}\in\M_{2g-1}$ with $\iota:\tilde{C}\to\tilde{C}$ a fixed point free involution.
\end{enumerate}
\end{lemma}

\begin{proof}
Given a point of order two, we get an unramified double cover by looking at $\underline{\Spec}(\sh{O}_C\oplus \mu)$.  Conversely, given a double cover, there is a single point of order two that pulls back to zero.

To get between 1 and 3, we note that $C\cong \tilde{C}/\iota$.
\end{proof}

By Lemma \ref{lemma:prymdefs}, we have an equivalence between $\J(C)[2]$ and the set of \'etale double covers, with 0 corresponding to the trivial double cover.  This induces a group structure on double covers, and if $\tilde{C}_\mu,\tilde{C}_\nu$ correspond to $\mu,\nu$, then $\tilde{C}_{\mu+\nu}$ is given by $\tilde{C}_\mu\times_C\tilde{C}_\nu/(\iota_\mu,\iota_\nu)$.

\begin{definition}[quasistable curve]
A genus $g\geq 2$ curve $X$ is \emph{quasistable} if every smooth rational component has at least two nodes and no two of these components intersect. We call the stable curve $C$ obtained by removing these rational components and gluing the nodes together the stabilization of $X$, and the nodes of $C$ obtained this way are the exceptional nodes and the rational components of $X$ over them exceptional components.
\end{definition}

We define the nonexceptional curve to be the union of the nonexceptional components and denote it by $X_{ne}$.

\begin{definition}[Prym curve]
A \emph{Prym curve} is a triple $(X,\eta,\beta)$ where $X$ is quasistable, $\eta\in\J(X)$ such that for all exceptional components, $E$, we have $\eta_E\cong\sh{O}_E(1)$, and $\beta:\eta^{\otimes 2}\to\sh{O}_X$ is a homomorphism that is generically nonzero on each nonexceptional component.
\end{definition}

\begin{remark}
In the notation we will use for other objects, a Prym curve would be called a $\ZZ/2\ZZ$ curve, but we will continue to refer to them as Prym curves.
\end{remark}

\begin{definition}[Isomorphism of Prym Curves]
An isomorphism of Prym curves $(X,\eta,\beta)$ and $(X',\eta',\beta')$ is an isomorphism $\sigma:X\to X'$ such that there exists an isomorphism $\tau:\sigma^*(\eta')\to \eta$ such that the diagram commutes:

\begin{center}
\leavevmode
\begin{xy}
(0,0)*+{\sigma^*(\sh{O}_{X'})}="a";
(20,0)*+{\sh{O}_X}="b";
(0,20)*+{\sigma^*(\eta')^{\otimes 2}}="c";
(20,20)*+{\eta^{\otimes 2}}="d";
{\ar^{\sim} "a";"b"};
{\ar^{\tau^{\otimes 2}} "c";"d"};
{\ar_{\sigma^*(\beta')} "c";"a"};
{\ar^{\beta} "d";"b"};
\end{xy}
\end{center}
\end{definition}

Note that the definition of isomorphism of Prym curves does not depend on what $\tau$ is chosen, only on $\sigma$.

Any automorphism of a Prym curve which induces the identity on the stable model of $X$ will be called \textit{inessential}.  The group of automorphisms will be denoted by $\Aut(X,\eta,\beta)$ and the inessential automorphisms will be $\Aut_0(X,\eta,\beta)$.

For the rest of this paper, for any curve, denote by $\nu$ the normalization morphism for a curve, and denote by $g^\nu$ the geometric genus of the normalization:

\begin{proposition}[{\cite[Proposition 11]{MR2117416}}]
Let $X$ be a quasistable curve, $Z$ its stable model, $\Gamma_Z$ the dual graph of $Z$, and $\Delta_X$ the set of nodes not lying under an exceptional curve, and assume further that $\Delta_X^c$ is eulerian.  Then there are $2^{2g^\nu+b_1(\Delta_X)}$ Prym curves supported on $X$ and each has multiplicity $2^{b_1(\Gamma_Z)-b_1(\Delta_X)}$ in the fiber of $\R_g\to\M_g$.
\end{proposition}

We will denote the moduli space of these curves by $\overline{\ZM{g}}$, and we note that it has two components.  One, isomorphic to $\overline{\M}_g$ where the Prym curve has $\eta\cong \sh{O}_X$ over a stable base, and $\overline{\R}_g$, the nontrivial Prym curves.

\begin{remark}[Notation]
The boundary components of the Prym moduli space have several competing notations in the literature.  For the $2^{2g}$ objects, we always have 1 that is the disconnected double cover, and in this setting, it lies over the stable curve itself.

Additionally, Donagi \cite{MR903385} classified the nontrivial points of order two on an irreducible 1-nodal curve in terms of the vanishing cycle $\delta$.  In his notation, $\Delta_I$ is the subset with the marked point $\mu$ being equal to $\delta$, $\Delta_{II}$ when $\langle \delta,\mu\rangle=0$ but $\mu\neq\delta$ and $\Delta_{III}$ being when $\langle \delta,\mu\rangle\neq 0$, under the Weil pairing, defined below.  Alternately, these three components are denoted by $\Delta_0''$, $\Delta_0'$ and $\Delta_0^{ram}$ by Farkas \cite{MR2639318} and it is noted that $\Delta_{III}=\Delta_0^{ram}$ is precisely the set of Prym curves on the quasistable curve.  In the rest of this article, however, we will follow Donagi's notation.

Over the 1-nodal reducible curves, the notation agrees, and the components are $\Delta_i$, $\Delta_{g-i}$ and $\Delta_{i:g-i}$, for the Prym curves supported on the component of genus $i$, $g-i$ or both, respectively.
\end{remark}

\section{The space \texorpdfstring{$\overline{\ZZM{g}}$}{ZMg}}

\begin{definition}[Weil Pairing]
Let $\mu,\nu\in\J(C)[2]$, and let $\kappa\in\Pic(C)$ such that $\kappa^{\otimes 2}\cong K_C$.  Then the Weil pairing on the curve $C$ is given by \[\langle\mu,\nu\rangle=h^0(\kappa)+h^0(\kappa\otimes\mu)+h^0(\kappa\otimes \nu)+h^0(\kappa\otimes\mu\otimes\nu)\mod 2.\]
\end{definition}

The Weil pairing is bilinear, skew-symmetric, and independent of the choice of $\kappa$, and therefore we can see that $\langle\mu,\nu\rangle=\langle\mu,\mu+\nu\rangle=\langle \nu,\mu+\nu\rangle$, and so it is an invariant of a rank 2 subgroup of $\J(C)[2]$.

\begin{lemma}
\label{lemma:Z22Mg}
The following data are equivalent:

\begin{enumerate}
 \item A curve $C\in\M_g$ along with $\tilde{C}_i\to C$, $i=1,2$ irreducible, nonisomorphic unramified double covers,
 \item A curve $C\in\M_g$ along with $\mu_1,\mu_2\in \J(C)[2]$, with $\mu_1\neq\mu_2$, and
 \item A pair of curves $C,D\in \M_{2g-1}$ along with involutions $\sigma_C,\sigma_D$ that act freely on $C,D$ respectively and which are such that $C/\sigma_C$ and $D/\sigma_D$ are isomorphic, but the pairs $(C,\sigma_C)$ and $(D,\sigma_D)$ are not.
\end{enumerate}
\end{lemma}

\begin{proof}
This is just an application of Lemma \ref{lemma:prymdefs}.
\end{proof}

A related, but slightly different result is the following, where we do not choose a basis for the Klein four group.

\begin{lemma}
\label{lemma:K2^2Mg}
The following data are equivalent:

\begin{enumerate}
 \item A curve $C\in \M_g$, and $\tilde{C}\to C$ an \'etale Klein $4$ cover,
 \item A curve $C\in \M_g$, and $\phi:V_{4}\to \J(C)[2]$ an injective homomorphism, and 
 \item A curve $\tilde{C}\in \M_{4g-3}$ with a free action of $V_{4}$ on $\tilde{C}$.
\end{enumerate}
\end{lemma}

\begin{proof}
This is again just an application of Lemma \ref{lemma:prymdefs}.
\end{proof}

For the rest of this section, we will be working in the case with a basis, and later will return to the basis-free case.

\begin{definition}[$\ZZ_2^2$ curve]
A $\ZZ_2^2$ curve is $(X_1,X_2,\eta_1,\eta_2,\beta_1,\beta_2)$ where $(X_i,\eta_i,\beta_i)$ is a Prym curve for $i=1,2$ and $X_1$ and $X_2$ have isomorphic stabilizations.
\end{definition}

An isomorphism of $\ZZ_2^2$ curves is just a pair of isomorphisms of Prym curves that induce the same isomorphism on the stable model.  Equivalently, an automorphism can be seen by looking at the quasistable curve with exceptional nodes the union of the exceptional nodes of $X_1$ and $X_2$.  From this viewpoint, an isomorphism is an isomorphism of such quasi-stable curves which induces isomorphisms on the pullbacks of $\eta_i$ and $\eta_i'$.

Given this, we can see that the moduli space of $\ZZ_2^2$ curves is just $\overline{\ZM{g}}\times_{\overline{\M}_g}\overline{\ZM{g}}$.  It is easy to see that it has at least five components, depending on $\eta_1,\eta_2$.  If both are trivial, we have a copy of $\overline{\M}_g$.  If one is trivial but the other is nontrivial, then we get two copies of $\R_g$.  Additionally, when $\eta_1\cong \eta_2$, we get another copy of $\overline{\R}_g$ leaving the remnant:

\begin{definition}[$\overline{\ZZM{g}}$]
We denote by $\overline{\ZZM{g}}$ the closure in $\overline{\ZM{g}}\times_{\overline{\M}_g}\overline{\ZM{g}}$ of the locus of $\ZZ_2^2$ curves with $\eta_1\not\cong\eta_2$ both nontrivial.  The locus of smooth curves in $\overline{\ZZM{g}}$ will be denoted by $\ZZM{g}$.
\end{definition}

Geometrically, the most natural thing to study is the moduli space of Klein four subgroups, with no choice of basis.  This, by Lemma \ref{lemma:K2^2Mg} is then the moduli space of Klein four covers of curves.  However, it is much easier to construct the space with a choice of basis, which is $\overline{\ZZM{g}}$ (also note that the geometricity of the covers is less clear on the boundary).  There is a natural $\PSL_2(\FF_2)$ action on $\ZZM{g}$, permuting the ordered bases of the $\ZZ_2^2\subset \J(C)[2]$.  Below, we will extend it to the boundary and construct $\overline{\KM{g}}$, and we will deduce many of its properties from those of $\overline{\ZZM{g}}$.

Here, we note that $\Aut(C,\eta_1,\eta_2)$ is just $\Aut(X_1,\eta_1,\beta_1)\times_{\Aut(C)}\Aut(X_2,\eta_2,\beta_2)$ where $(X_i,\eta_i,\beta_i)$ are the two Prym structures. Thus, the results of \cite{MR2117416} apply without difficulty, so we have good behavior of the universal deformation, we have a subgroup of inessential automorphisms, etc.  We will recall those facts we need in the last section of the paper.

\begin{lemma}
\label{DegZZM}
The map $\ZZM{g}\to\M_g$ has degree $(2^{2g}-1)(2^{2g}-2)$, thus, so does $\overline{\ZZM{g}}\to\overline{\M_g}$.
\end{lemma}

\begin{proof}
We know that the degree of $\ZM{g}\times_{\M_g}\ZM{g}\to\M_g$ is $2^{4g}$ and that it breaks down as $\M_g\cup\R_g\cup\R_g\cup\R_g\cup\ZZM{g}$.  Each component is dominant and equidimensional, so we can compute the degree on $\ZZM{g}$ as $2^{4g}-1-3(2^{2g}-1)=(2^{2g}-1)(2^{2g}-2)$.
\end{proof}

The space $\ZZM{g}$ is not irreducible.  We can see that there must be at least two components because the Weil pairing is deformation invariant; we will write $\ZZM{g}^0$ and $\ZZM{g}^1$ for the two components, with Weil pairing respectively 0 and 1.

\begin{lemma}
The spaces $\ZZM{g}^0$ and $\ZZM{g}^1$ are both irreducible.
\end{lemma}

\begin{proof}
For any curve $C$ of genus $g$, the action of $\Sp(2g,\FF_2)$ on the space $\J(C)[2]\times \J(C)[2]$ has two orbits, pairs of points that are orthogonal and pairs that are nonorthogonal, and this is the monodromy of $\ZZM{g}\to \M_g$.
\end{proof}

\begin{remark}
Although slightly more complex, in the case where we look at points of order $n$ rather than points of order $2$, we get a similar theorem, where the number of components is indexed by $\ZZ_n$.  Similarly, the Klein moduli space, which we will study below, will have components indexed by $\ZZ_n/\ZZ_n^\times$.  In our case, both of these are $\ZZ_2$, and we will identify them with $\{0,1\}$.
\end{remark}

Before moving on to a detailed analysis of the boundary, we compute the degrees of the maps $\ZZM{g}^i\to\M_g$.

\begin{proposition}
\label{DegZZM01}
We have natural maps to $\M_g$ forgetting the points of order two, and their degrees are:

\begin{enumerate}
 \item $\ZZM{g}^0\to \M_g$ has degree $(2^{2g}-1)(2^{2g-1}-2)$
 \item $\ZZM{g}^1\to \M_g$ has degree $(2^{2g}-1)2^{2g-1}$
\end{enumerate}
\end{proposition}

\begin{proof}
Fix a smooth curve $C$ of genus $g$.

Any element of $\ZZM{g}$ lying over $C$ is of the form $(C,\eta,\eta')$ where $\eta,\eta'\in\J(C)[2]$, which is an $\FF_2$ vector space. To be in $\ZZM{g}^0$ they must be orthogonal under the Weil pairing, which is a nondegenerate form.  Thus, for each $\eta$, we must choose $\eta'\in\eta\perp$, a hyperplane in $\J(C)[2]$.  However, as they are linearly independent, we asset that $\eta'\notin\{0,\eta\}$.  Thus, we have $2^{2g}-1$ choices for $\eta$, and given one of those, we have $2^{2g-1}-2$ choices of $\eta'$ satisfying these conditions, which computes the degree of $\ZZM{g}^0\to\M_g$.

To determine the degree of $\ZZM{g}^1\to \M_g$, we note that we can again choose $\eta$ freely, and now $\eta'\in \J(C)[2]\setminus \eta^\perp$, giving us the degree claimed.
\end{proof}

\section{Geometry of the boundary}

\begin{proposition}
Let $Z$ be a stable curve with dual graph $\Gamma_Z$.  The fiber over $Z$ in the $\overline{\ZZM{g}}$ consists of the following objects:

\begin{enumerate}
 \item $(2^{2g^\nu+b_1(\Gamma_Z)}-1)(2^{2g^\nu+b_1(\Gamma_Z)}-2)$ objects of multiplicity 1 where both Prym curves are supported on $Z$.
 \item If $X$ is quasistable with stabilization $Z$ and $\Delta_X^c$ is Eulerian, then we get $2^{4g^\nu+b_1(\Delta_X)+b_1(\Gamma_Z)}-2^{2g^\nu+b_1(\Delta_X)}$ objects of multiplicity $2^{b_1(\Gamma_Z)-b_1(\Delta_X)}$ supported on each of $(X,Z)$ and $(Z,X)$.
 \item If $X$ is quasistable with stabilization $Z$ and $\Delta_X^c$ is Eulerian, then we get two types of objects supported on $(X,X)$:
		\begin{enumerate}
		 \item $2^{4g^\nu+2b_1(\Delta_X)}-2^{2g^\nu+b_1(\Delta_X)}$ objects of multiplicity $2^{2b_1(\Gamma_Z)-2b_1(\Delta_X)}$ with two distinct Prym structures
		 \item $2^{2g^\nu+b_1(\Delta_X)}$ objects of multiplicity $2^{2b_1(\Gamma_Z)-2b_1(\Delta_X)}-2^{b_1(\Gamma_Z)-b_1(\Delta_X)}$ with the same Prym structure.
		\end{enumerate}
 \item If $X_1$, $X_2$ are quasistable over $Z$ and $\Delta_{X_1}^c$ and $\Delta_{X_2}^c$ are Eulerian, then there exist $2^{4g^\nu+b_1(\Delta_X)+b_1(\Delta_X)}$ objects of multiplicity $2^{2b_1(\Gamma_Z)-b_1(\Delta_{X_1})-b_1(\Delta_{X_2})}$ on $(X_1,X_2)$ and again on $(X_2,X_1)$.
\end{enumerate}
\end{proposition}

\begin{proof}
The fiber is a subscheme of $R_Z\times R_Z$ which is the complement of $M_g\cup R_g\cup R_g\cup R_g$, components which split off completely over smooth base curves.  Away from the diagonal, we simply remove anything where one of the components is the trivial line bundle.  On the diagonal, the situation is somewhat more intricate.  We look at $\overline{\ZM{g}}\times_{\overline{\M_g}}\overline{\ZM{g}}\setminus(\overline{\M_g}\cup\overline{\ZM{g}}\cup\overline{\ZM{g}}\cup\overline{\ZM{g}})$ and then take the closure.  This leaves us with some points on the diagonal, which we can see because on any degeneration, there will be classes of Prym curves that will degenerate to the same thing, which is seen by noting the multiplicity greater than 1.

Additionally, a straightforward computation summing these over all of the Euler paths gives us total degree $(2^{2g^\nu+2b_1(\Gamma_Z)}-1)(2^{2g^\nu+2b_1(\Gamma_Z)}-2)$, which is the degree of the moduli space, showing that nothing has been missed.
\end{proof}

Now, applying the above to the general point on the boundary, we see that for a reducible 1-nodal curve, all Prym curves on it are supported on the stable curve itself, yielding $(2^{2g}-1)(2^{2g}-2)$ objects of multiplicity 1. The case of an irreducible curve is a bit more complex:

\begin{corollary}
\label{deltairrlemma}

Let $Z$ be a 1-nodal irreducible stable curve of genus $g$ and $\nu:Z^\nu\to Z$ its normalization.  There is only one unstable quasistable curve, $X=Z^\nu\cup_{x,y}\PP^1$, where $x,y$ are the preimages of the node in $Z^\nu$.  Then the fiber of $\overline{\ZZM{g}}\to \overline{\M}_g$ over $Z$ consists of the following objects:

\begin{enumerate}
 \item On $(Z,Z)$, we have $(2^{2g-1}-1)(2^{2g-1}-2)$ objects of multiplicity 1.
 \item On each $(Z,X)$ and $(X,Z)$, we have a total of $2^{4g-2}-2^{2g-1}$ objects of multiplicity 2
 \item On $(X,X)$ we have two types of objects:
\begin{enumerate}
 \item $2^{4g-4}-2^{2g-2}$ objects of multiplicity 4 with non-isomorphic projections to $\overline{\ZM{g}}$.
 \item $2^{2g-2}$ objects of multiplicity 2 with the same projections to $\overline{\ZM{g}}$.
\end{enumerate}
\end{enumerate}
\end{corollary}

Now, we must compute the list of boundary divisors.  We begin by looking at the boundary of $\overline{\ZZM{g}}\times_{\overline{\M_g}}\overline{\ZZM{g}}$.  Points on the boundary can be classified into products $\Delta_a\times \Delta_b$ where $a,b$ were in $\{I,II,III\}$ over an irreducible 1-nodal curve and $\{i,g-i,i:g-i\}$ over a reducible curve with components of genus $i$ and $g-i$.  We denote the restrictions of these loci to $\overline{\ZZM{g}}$ by $\Delta_{a,b}$, and note that although some are, many of these are not irreducible.  

The components with nonirreducible restrictions to $\overline{\ZZM{g}}^i$ are $\Delta_{II,II}$, $\Delta_{III,III}$, $\Delta_{i,i:g-i}$, $\Delta_{i:g-i,i}$, $\Delta_{g-i,i:g-i}$, $\Delta_{i:g-i,g-i}$, and $\Delta_{i:g-i,i:g-i}$.  Specifically, they break up as
\begin{eqnarray*}
\Delta_{II,II}				&=&	\Delta_{II,II}^{\pm}+\Delta_{II,II}'\\
\Delta_{III,III}			&=&	\Delta_{III,III}^{diag}+\Delta_{III,III}'\\
\Delta_{i,i:g-i}			&=&	\Delta_{i,i:g-i}^i+\Delta_{i,i:g-i}'\\
\Delta_{i:g-i,i}			&=&	\Delta_{i:g-i,i}^i+\Delta_{i:g-i,i}'\\
\Delta_{g-i,i:g-i}		&=&	\Delta_{g-i,i:g-i}^{g-i}+\Delta_{g-i,i:g-i}'\\
\Delta_{i:g-i,g-i}		&=&	\Delta_{i:g-i,g-i}^{g-i}+\Delta_{i:g-i,g-i}'\\
\Delta_{i:g-i,i:g-i}	&=&	\Delta_{i:g-i,i:g-i}^i+\Delta_{i:g-i,i:g-i}^{g-i}+\Delta_{i:g-i,i:g-i}'
\end{eqnarray*}
where (with equality meaning equal to the closure of)
\begin{eqnarray*}
\Delta_{II,II}^{\pm}			&=&	\{(\eta_1,\eta_2)|\nu^*\eta_1\cong\nu^*\eta_2\}\\
\Delta_{II,II}'						&=&	\{(\eta_1,\eta_2)|\nu^*\eta_1\not\cong\nu^*\eta_2\}\\
\Delta_{III,III}^{diag}		&=&	\{(\eta_1,\eta_2)|\eta_1\cong \eta_2\}\\
\Delta_{III,III}'					&=&	\{(\eta_1,\eta_2)|\eta_1\not\cong\eta_2\}\\
\Delta_{i,i:g-i}^i				&=&	\{(\eta_1,\eta_2)|\eta_1|_{C_1}\cong \eta_2|_{C_1}\}\\
\Delta_{i,i:g-i}'					&=&	\{(\eta_1,\eta_2)|\eta_1|_{C_1}\not\cong \eta_2|_{C_1}\}\\
\Delta_{i:g-i,i}^i				&=&	\{(\eta_1,\eta_2)|\eta_1|_{C_1}\cong \eta_2|_{C_1}\}\\
\Delta_{i:g-i,i}'					&=&	\{(\eta_1,\eta_2)|\eta_1|_{C_1}\not\cong \eta_2|_{C_1}\}\\
\Delta_{g-i,i:g-i}^{g-i}	&=&	\{(\eta_1,\eta_2)|\eta_1|_{C_2}\cong \eta_2|_{C_2}\}\\
\Delta_{g-i,i:g-i}'				&=&	\{(\eta_1,\eta_2)|\eta_1|_{C_2}\not\cong \eta_2|_{C_2}\}\\
\Delta_{i:g-i,g-i}^{g-i}	&=&	\{(\eta_1,\eta_2)|\eta_1|_{C_2}\cong \eta_2|_{C_2}\}\\
\Delta_{i:g-i,g-i}'				&=&	\{(\eta_1,\eta_2)|\eta_1|_{C_2}\not\cong \eta_2|_{C_2}\}\\
\Delta_{i:g-i,i:g-i}^i		&=&	\{(\eta_1,\eta_2)|\eta_1|_{C_1}\cong \eta_2|_{C_1}\}\\
\Delta_{i:g-i,i:g-i}^{g-i}&=&	\{(\eta_1,\eta_2)|\eta_1|_{C_2}\cong \eta_2|_{C_2}\}\\
\Delta_{i:g-i,i:g-i}'			&=&	\{(\eta_1,\eta_2)|\eta_1|_{C_i}\not\cong \eta_2|_{C_i}\mbox{ for }i=1,2\}
\end{eqnarray*}

Now that we have a list of all of the components, we compute which objects are in which, a straightforward computation:

\begin{proposition}
\label{prop:1nidegrees}
Over the locus of 1-nodal irreducible curves in $\overline{\M_g}$ of the boundary components of $\overline{\ZZM{g}}$ consist of

\begin{eqnarray*}
\Delta_{I,II}							&=&	2^{2g-1}-2\mbox{ objects of multiplicity }1\\
\Delta_{II,I}							&=&	2^{2g-1}-2\mbox{ objects of multiplicity }1\\
\Delta_{I,III}						&=&	2^{2g-2}\mbox{ objects of multiplicity }2\\
\Delta_{III,I}						&=&	2^{2g-2}\mbox{ objects of multiplicity }2\\
\Delta_{II,III}						&=&	2^{2g-2}(2^{2g-1}-2)\mbox{ objects of multiplicity }2\\
\Delta_{III,II}						&=&	2^{2g-2}(2^{2g-1}-2)\mbox{ objects of multiplicity }2\\
\Delta_{II,II}^{\pm}			&=&	2^{2g-1}-2\mbox{ objects of multiplicity }1\\
\Delta_{II,II}'						&=&	(2^{2g-1}-2)(2^{2g-1}-4)\mbox{ objects of multiplicity }1\\
\Delta_{III,III}^{diag}		&=&	2^{2g-2}\mbox{ objects of multiplicity }2\\
\Delta_{III,III}'					&=&	2^{4g-4}-2^{2g-2}\mbox{ objects of multiplicity }4
\end{eqnarray*}
\end{proposition}

\begin{remark}
\label{Intersection}
It is interesting to note that away from $\Delta_{III,III}'$, the Weil pairing is well-defined by continuity.  However, on this divisor, this fails.  Let us look at a simple example.  Let $C$ be a smooth genus 2 curve with $p_1,\ldots,p_6$ the fixed points of the hyperelliptic involution.  Then any point of order two is $p_i-p_j$ for $i\neq j$, and $p_i-p_j\equiv p_j-p_i$.  If the vanishing cycle of the degeneration is $p_1-p_2$, then the pairs $(p_1-p_3,p_1-p_4)$ and $(p_1-p_3,p_2-p_4)$ both degenerate to the same point of $\Delta_{III,III}'$.  We note that, for a genus 2 curve, the Weil pairing can be described as the cardinality of the intersection of the set of indices appearing in this representation.  Thus, $\langle p_1-p_3,p_1-p_4\rangle=1$ and $\langle p_1-p_3,p_2-p_4\rangle=0$.  Thus, the two components of this moduli space intersect on the boundary!
\end{remark}

In fact, we can say a bit more about the intersection:

\begin{theorem}
\label{thm:intersection}
The intersection $\Delta_{III,III}'=\overline{\ZZM{g}}^0\cap\overline{\ZZM{g}}^1$ is transverse, in the sense that if $(C,\eta_1,\eta_2)\in\Delta_{III,III}'$, then we have \[T_{(C,\eta_1,\eta_2)}\overline{\ZZM{g}}\cong T_{(C,\eta_1,\eta_2)}\overline{\ZZM{g}}^0\oplus_{T_{(C,\eta_1,\eta_2)}\Delta_{III,III}'}T_{(C,\eta_1,\eta_2)}\overline{\ZZM{g}}^1.\]
\end{theorem}

\begin{proof}
To see that this is precisely the intersection, we look at a degeneration of Prym curves with vanishing cycle $\delta$.  The fiber over the nodal curve only has points coming together over $\Delta_{III}$, which is the part of the fiber over the quasi-stable curve.  There, every Prym curve structure is the limit of both $\eta$ and $\eta+\delta$ for $\eta$ some Prym curve structure on a smooth curve in the degeneration.  The Weil pairing is well-defined on all components over $\Delta_{III}$ other than $\Delta_{III,III}'$, by linearity and the definitions of $\Delta_I$ and $\Delta_{II}$.  However, on $\Delta_{III,III}$, we have $(\eta_1,\eta_2)$ two Prym structures giving a $\ZZ_2^2$ curve.  This limit as $\delta$ vanishes, is the same as the limit of $(\eta_1,\eta_2+\delta)$, and also two other loci.  However, the Weil pairing is linear, so $\langle\eta_1,\eta_2+\delta\rangle=\langle\eta_1,\eta_2\rangle+\langle\eta_1,\delta\rangle$, and because it is in $\Delta_{III}$, $\langle \eta_1,\delta\rangle=1$, so this point is a limit of families of Veil pairing both 0 and 1, and this holds for every point in $\Delta_{III,III}'$.

Transversality follows by looking at first order deformations of $(C,\eta_1,\eta_2)$.  The Weil pairing determines which component $(C,\eta_1,\eta_2)$ is on, except along $\Delta_{III,III}'$ where it is indeterminate.  So we describe all of the space:

\begin{eqnarray*}
T_{(C,\eta_1,\eta_2)}\overline{\ZZM{g}}			&=&	\mbox{First order deformations with }\langle\eta_1,\eta_2\rangle\mbox{ undefined, } 0, \mbox{ or }1,	\\
 T_{(C,\eta_1,\eta_2)}\overline{\ZZM{g}}^0	&=&	\mbox{First order deformations with }\langle\eta_1,\eta_2\rangle\mbox{ undefined, or } 0,	\\
{T_{(C,\eta_1,\eta_2)}\Delta_{III,III}'}		&=&	\mbox{First order deformations with }\langle\eta_1,\eta_2\rangle\mbox{ undefined,}	\\
T_{(C,\eta_1,\eta_2)}\overline{\ZZM{g}}^1		&=&	\mbox{First order deformations with }\langle\eta_1,\eta_2\rangle\mbox{ undefined, or } ,
\end{eqnarray*}

From these descriptions, transversality follows immediately.
\end{proof}

A similar computation to Proposition \ref{prop:1nidegrees} over reducible 1-nodal curves gives:

\begin{proposition}
\label{prop:1nrdegrees}
Over the locus of 1-nodal reducible curves that are a union of a genus $i$ and a genus $g-i$ curve in $\overline{\M_g}$, the boundary components of $\overline{\ZZM{g}}$ are of the following degrees (and all objects are multiplicity 1):
\begin{eqnarray*}
\Delta_{i,i}							&=& (2^{2i}-1)(2^{2i}-2)\\
\Delta_{g-i,g-i}					&=& (2^{2(g-i)}-1)(2^{2(g-i)}-2)\\
\Delta_{i,g-i}						&=& (2^{2i}-1)(2^{2(g-i)}-1)\\
\Delta_{g-i,i}						&=& (2^{2i}-1)(2^{2(g-i)}-1)\\
\Delta_{i,i:g-i}^i				&=&	(2^{2i}-1)(2^{2(g-i)}-1)\\
\Delta_{i,i:g-i}'					&=&	(2^{2i}-2)(2^{2i}-1)(2^{2(g-i)}-1)\\
\Delta_{i:g-i,i}^i				&=&	(2^{2i}-1)(2^{2(g-i)}-1)\\
\Delta_{i:g-i,i}'					&=&	(2^{2i}-2)(2^{2i}-1)(2^{2(g-i)}-1)\\
\Delta_{g-i,i:g-i}^{g-i}	&=&	(2^{2i}-1)(2^{2(g-i)}-1)\\
\Delta_{g-i,i:g-i}'				&=&	(2^{2(g-i)}-2)(2^{2i}-1)(2^{2(g-i)}-1)\\
\Delta_{i:g-i,g-i}^{g-i}	&=&	(2^{2i}-1)(2^{2(g-i)}-1)\\
\Delta_{i:g-i,g-i}'				&=&	(2^{2(g-i)}-2)(2^{2i}-1)(2^{2(g-i)}-1)\\
\Delta_{i:g-i,i:g-i}^i		&=&	(2^{2i}-1)(2^{2(g-i)}-1)(2^{2(g-i)}-2)\\
\Delta_{i:g-i,i:g-i}^{g-i}&=&	(2^{2i}-1)(2^{2i}-2)(2^{2(g-i)}-1)\\
\Delta_{i:g-i,i:g-i}'			&=&	(2^{2i}-1)(2^{2(g-i)}-1)\\&&\times((2^{2i}-1)(2^{2(g-i)}-1)-(2^{2i}-1)-(2^{2(g-i)}-1))
\end{eqnarray*}
\end{proposition}

\section{The moduli of Klein curves}

In this section, we extend the results of the previous section to a compactification of $\KM{g}=\ZZM{g}/\PSL_2(\FF_2)$, that is, the moduli space where we do not choose a basis of the Deck transformations.  We will do so by extending the group action to the compactification $\overline{\ZZM{g}}$, and defining the quotient to be $\overline{\KM{g}}$.

\begin{proposition}
The group action $\PSL_2(\FF_2)\times \ZZM{g}\to\ZZM{g}$ extends to each component of $\overline{\ZZM{g}}$.
\end{proposition}

\begin{proof}
Let $D_1=\Delta_{I,III}\cup\Delta_{III,I}\cup\Delta_{III,III}^{\diag}$ and $D_2=\Delta_{III,III}'\cup \Delta_{II,III}\cup\Delta_{III,II}$.  Then the extension is actually straightforward over $\overline{\ZZM{g}}\setminus (D_1\cup D_2)$, as over this locus, the fibers are reduced, and the action is just by change of basis on a Klein 4 group.  Only in the cases where multiplicities are no longer 1, namely $D_1$ and $D_2$, will these fail to just be Klein four groups.

Now, we take the orbit of a $\ZZ_2^2$-curve in the locus where we have the group action, and degenerate it to $D_1$.  Then, the six objects of multiplicity 1 of this fiber degenerate to $(\eta,\eta)\in\Delta_{III,III}^{\diag}$, $(\eta,\sh{O}^-_X)\in \Delta_{III,I}$ and $(\sh{O}^-_X,\eta)\in\Delta_{I,III}$, where $\sh{O}^-_X$ is the Prym curve structure on $X$ lying in $\Delta_I\subset\overline{\ZM{g}}$.  Each of these appears with multiplicity 2.  Here, the group action can be seen most clearly by noting that $\PSL_2(\FF_2)\cong S_3$ (and in fact, the change of basis on a Klein four group is just permuting the three nonzero elements) and seeing the action as being that of $S_3$ on the ordered set $(\sh{O}_X^-,\eta,\eta)$ followed by forgetting the last element.  Deeper degeneration into the strata $D_1\setminus D_2$ can be handled in the same way, leaving only $D_2$ remaining.

To extend the action to $D_2$, we must first restrict to the individual irreducible components of $\overline{\ZZM{g}}$.  This is because $\Delta_{III,III}'$ is the intersection of the two components, by Theorem \ref{thm:intersection}.  So, by transversality the fiber multiplicity of elements in the intersection must be split evenly between the components.

Now, let $(\epsilon_1,\epsilon_2)\in\Delta_{III,III}'$.  Then $\epsilon_2\otimes\epsilon^{-1}_1$ gives a Prym structure on the closure of the complement of the exceptional components.  There are two different Prym structures on the stable curve that pull back to this under the stabilization map, but one of them lies on each componenet.  We will denote by $\eta^i$ the Prym structure such that $(\eta^i,\epsilon_1),(\eta^i,\epsilon_2)\in\overline{\ZZM{g}}^i$.  Then, the $\PSL_2(\FF_2)\cong S_3$ action is given by permutation of $\epsilon_1,\epsilon_2,\eta^i$.
\end{proof}

This extension allows us to take the quotient, which constructs from the moduli of pairs of Prym structures, which is the same as the moduli of Klein four groups of Prym structures, with the moduli of Klein four groups of Prym covers without a basis.

\begin{definition}[Moduli of Klein four covers]
We define the space $\overline{\KM{g}}$ to be the quotient of $\overline{\ZZM{g}}$ by the relation described above, and we call it the \emph{moduli of Klein four covers of genus $g$ curves}.
\end{definition}

Given an orbit $\{(C,\eta_1^i,\eta_2^i)\}$ where $i$ runs over the elements of the orbit, we will denote by $(C,\{\eta^i_j\}_{i,j})$ the corresponding point of $\overline{\KM{g}}$, with $i$ running over the orbit and $j=1,2$.

The boundary in the Klein moduli space simplifies significantly.  Because the action of $\PSL_2(\FF_2)$ exchanges some boundary components, we group them together and give names to their images (fixing the Weil pairing as either $0$ or $1$ in each case) in the following:

\begin{eqnarray*}
 \Delta_{I,II}\cup \Delta_{II,I}\cup\Delta_{II,II}^\pm&\to&\Delta_{I,II,II}\\
 \Delta_{I,III}\cup\Delta_{III,I}\cup\Delta_{III,III}^{diag}&\to&\Delta_{I,III,III}\\
 \Delta_{II,III}\cup\Delta_{III,II}\cup\Delta_{III,III}'&\to&\Delta_{II,III,III}\\
 \Delta_{II}'&\to&\Delta_{II,II,II}\\
 \Delta_{i,g-i}\cup\Delta_{g-i,i}\cup\Delta_{i,i:g-i}^i\cup\Delta_{i:g-i,i}^i\cup\Delta_{g-i,i:g-i}^{g-i}\cup\Delta_{i:g-i,g-i}^{g-i}&\to&\Delta_{i,g-i,i:g-i}\\
 \Delta_{i,i}&\to&\Delta_{i,i,i}\\
 \Delta_{g-i,g-i}&\to&\Delta_{g-i,g-i,g-i}\\
 \Delta_{i,i:g-i}'\cup \Delta_{i:g-i,i}' \cup\Delta_{i:g-i,i:g-i}^i&\to&\Delta_{i,i:g-i,i:g-i}\\
 \Delta_{g-i,i:g-i}'\cup\Delta_{i:g-i,g-i}'\cup\Delta_{i:g-i,i:g-i}^{g-i}&\to&\Delta_{g-i,i:g-i,i:g-i}\\
 \Delta_{i:g-i,i:g-i}'&\to&\Delta_{i:g-i,i:g-i,i:g-i}
\end{eqnarray*}

Between the degrees computed in the previous section and the maps above all being $\PSL_2(\FF_2)$ quotients, we find the following structure on the boundary

\begin{eqnarray*}
 \Delta_{I,II,II}						&=&	2^{2g-2}-1\mbox{ objects of multiplicity }1\\
 \Delta_{I,III,III}					&=&	2^{2g-2}\mbox{ objects of multiplicity }1\\
 \Delta_{II,III,III}				&=&	2^{2g-2}(2^{2g-2}-1)\mbox{ objects of multiplicity }2\\
 \Delta_{II,II,II}					&=&	\frac{(2^{2g-1}-2)(2^{2g-1}-4)}{6}\mbox{ objects of multiplicity }1\\
 \Delta_{i,g-i,i:g-i}				&=&	(2^{2i}-1)(2^{2(g-i)}-1)\mbox{ objects of multiplicity }1\\
 \Delta_{i,i,i}							&=&	\frac{(2^{2i}-1)(2^{2i}-2)}{6}\mbox{ objects of multiplicity }1\\
 \Delta_{g-i,g-i,g-i}				&=&	\frac{(2^{2(g-i)}-1)(2^{2(g-i)}-2)}{6}\mbox{ objects of multiplicity }1\\
 \Delta_{i,i:g-i,i:g-i}			&=&	\frac{(2^{2i}-1)(2^{2i}-2)(2^{2(g-i)}-1)}{2}\mbox{ objects of multiplicity }1\\
 \Delta_{g-i,i:g-i,i:g-i}		&=&	\frac{(2^{2(g-i)}-1)(2^{2(g-i)}-2)(2^{2i}-1)}{2}\mbox{ objects of multiplicity }1\\
 \Delta_{i:g-i,i:g-i,i:g-i}	&=&	\frac{(2^{2i}-1)(2^{2(g-i)}-1)(2^{2i}-2)(2^{2(g-i)}-2)}{6}\mbox{ objects of}\\
&& \mbox{multiplicity }1
\end{eqnarray*}

This gives us

\begin{proposition}
The morphsim $\overline{\KM{g}}\to\M_g$ has degree $\frac{(2^{2g}-1)(2^{2g}-2)}{6}$ and is simply ramified along $\Delta_{II,III,III}$.
\end{proposition}

Thus, we have

\begin{corollary}
The canonical divisor of $\overline{\KM{g}}$ is 
\begin{eqnarray*}
K_{\overline{\KM{g}}}&=&13\lambda-2\Delta_{I,II,II}-2\Delta_{I,III,III}-2\Delta_{II,II,II}-3\Delta_{II,III,III}\\
&&-\Delta_{1,g-1,1:g-1}-\Delta_{1,1,1}-\Delta_{g-1,g-1,g-1}\\
&&-\Delta_{1,1:g-1,1:g-1}-\Delta_{g-1,1:g-1,1:g-1}-\Delta_{1:g-1,1:g-1,1:g-1}\\
&&-2\sum_{i=1}^{\lfloor g/2\rfloor}(\Delta_{i,g-i,i:g-i}+\Delta_{i,i,i}+\Delta_{g-i,g-i,g-i}+\Delta_{i,i:g-i,i:g-i}\\
&&+\Delta_{g-i,i:g-i,i:g-i}+\Delta_{i:g-i,i:g-i,i:g-i}).
\end{eqnarray*}
\end{corollary}

\begin{proof}
We use the Hurwitz formula, which tells us that $K_{\overline{\KM{g}}}=\pi^*K_{\overline{\M_g}}+\Delta_{II,III,III}$.  The canonical divisor of $\overline{\M_g}$ is $13\lambda-2\delta_0-3\delta_1-2\delta_2-\ldots -2\delta_{\lfloor g/2\rfloor}$\cite{MR664324}.

We note that $\pi^*(\Delta_i)=\Delta_{i,g-i,i:g-i}+\Delta_{i,i,i}+\Delta_{g-i,g-i,g-i}+\Delta_{i,i:g-i,i:g-i}+\Delta_{g-i,i:g-i,i:g-i}+\Delta_{i:g-i,i:g-i,i:g-i}$ and $\pi^*\Delta_0=\Delta_{I,II,II}+\Delta_{I,III,III}+\Delta_{II,II,II}+2\Delta_{II,III,III}$ and $\pi^*\lambda=\lambda$.
\end{proof}

Now we'll work out some numerics of the odd and even components of $\overline{\KM{g}}$

\begin{proposition}
The degrees of the natural projection maps to $\overline{\M_g}$ are
\begin{itemize}
 \item for $\overline{\KM{g}}$, 	$\frac{(2^{2g}-1)(2^{2g}-2)}{6}$.
 \item for $\overline{\KM{g}}^0$, $\frac{(2^{2g}-1)(2^{2g-1}-2)}{6}$.
 \item for $\overline{\KM{g}}^1$, $\frac{(2^{2g}-1)2^{2g-1}}{6}$.
\end{itemize}
\end{proposition}

This follows directly from Proposition \ref{DegZZM01} and Lemma \ref{DegZZM}.

As a final computation involving the degrees, we compute the content of the boundary divisors when restricted to $\overline{\KM{g}}^0$ and $\overline{\KM{g}}^1$.

\begin{theorem}
The fiber over the generic element $C$ of the boundary of $\overline{\M}_g$ in $\overline{\KM{g}}$ is:

\begin{enumerate}
 \item if $C$ is irreducible 1-nodal, the fiber in $\overline{\KM{g}}^0$ is 
	\begin{enumerate}
		\item $2^{2g-2}-1$ elements of $\Delta_{I,II,II}^0$ with multiplicity $1$,
		\item $\frac{(2^{2g-1}-2)(2^{2g-2}-4)}{6}$ elements of $\Delta_{II,II,II}^0$ with multiplicity $1$,
		\item ${(2^{2g-1}-2)2^{2g-4}}$ elements of $\Delta_{II,III,III}^0$ with multiplicity $2$,
	\end{enumerate}
 \item if $C$ is irreducible 1-nodal, the fiber in $\overline{\KM{g}}^1$ is 
	\begin{enumerate}
		\item $\frac{(2^{2g-1}-2)2^{2g-2}}{6}$ elements of $\Delta_{II,II,II}^1$ with multiplicity $1$,
		\item ${(2^{2g-1}-2)2^{2g-4}}$ elements of $\Delta_{II,III,III}^1$ with multiplicity $2$,
		\item ${2^{2g-2}}$ elements of $\Delta_{I,III,III}^1$ with multiplicity $1$,
	\end{enumerate}
 \item if $C$ is reducible with components of genus $i$ and $g-i$, the fiber in $\overline{\KM{g}}^0$ is 
	\begin{enumerate}
		\item $(2^{2i}-1)(2^{2(g-i)}-1)$ elements of $\Delta_{i,g-i,i:g-i}^0$ with multiplicity $1$,
		\item $\frac{(2^{2i}-1)(2^{2i-1}-2)}{6}$ elements of $\Delta_{i,i,i}^0$ with multiplicity $1$,
		\item $\frac{(2^{2(g-i)}-1)(2^{2(g-i)-1}-2)}{6}$ elements of $\Delta_{g-i,g-i,g-i}^0$ with multiplicity $1$,
		\item ${(2^{2i-1}-1)(2^{2i-2}-1)(2^{2(g-i)}-1)}$ elements of $\Delta_{i,i:g-i,i:g-i}^0$ with multiplicity $1$,
		\item ${(2^{2(g-i)-1}-1)(2^{2(g-i)-2}-1)(2^{2i}-1)}$ elements of $\Delta_{g-i,i:g-i,i:g-i}^0$ with multiplicity $1$,
		\item $\frac{(2^{2i}-1)(2^{2(g-i)}-1)((2^{2i-1}-2)(2^{2(g-i)-1}-2)+(2^{2i-1})(2^{2(g-i)-1}))}{6}$ elements of\\ $\Delta_{i:g-i,i:g-i,i:g-i}^0$ with multiplicity $1$,
	\end{enumerate}
 \item if $C$ is reducible with components of genus $i$ and $g-i$, the fiber in $\overline{\KM{g}}^1$ is 
	\begin{enumerate}
		\item $\frac{(2^{2i}-1)2^{2i-1}}{6}$ elements of $\Delta_{i,i,i}^1$ with multiplicity $1$,
		\item $\frac{(2^{2(g-i)}-1)2^{2(g-i)-1}}{6}$ elements of $\Delta_{g-i,g-i,g-i}^1$ with multiplicity $1$,
		\item ${(2^{2i-1}-1)(2^{2i-1})(2^{2(g-i)}-2)}$ elements of $\Delta_{i,i:g-i,i:g-i}^1$ with multiplicity $1$,
		\item ${(2^{2(g-i)-1}-1)(2^{2(g-i)-2})(2^{2i}-1)}$ elements of $\Delta_{g-i,i:g-i,i:g-i}^1$ with\\ multiplicity $1$,
		\item $\frac{(2^{2i}-1)(2^{2(g-i)}-1)((2^{2i-1}-2)(2^{2i-1})+(2^{2(g-i)-1})(2^{2(g-i)-1}-2))}{6}$ elements of\\ $\Delta_{i:g-i,i:g-i,i:g-i}^1$ with multiplicity $1$.
	\end{enumerate}
\end{enumerate}
\end{theorem}

\begin{proof}
We will work out parts 1 and 3, parts 2 and 4 being analagous.

With the exception of $\Delta_{II,III,III}$, we can compute the Weil pairing by choosing any pair of elements in the group.  For $\Delta_{II,III,III}$, we note that it must be divided evenly between the components.  The groups consist of two elements from $\Delta_{III}$ and the one element of $\Delta_{II}$, and can be chosen to either be glued by $+1$ or $-1$ at the node. These will correspond to Weil pairing $0$ and $1$, thus dividing $\Delta_{II,III,III}$ evenly.

For $\Delta_{I,II,II}$, by definition, we must have $\Delta_{I,II,II}^0=\Delta_{I,II,II}$.  Similarly, we can see that $\Delta_{I,III,III}^0=\emptyset$.  We can finish by computing that $\Delta_{II,II,II}^0$ must be the correct size to, with the other components, add up to $\frac{(2^{2g}-1)(2^{2g-1}-2)}{6}$.  However, we can compute this directly by choosing $\mu_1,\mu_2\in\Delta_{II}$ distinct and orthogonal.  Then, if $\delta$ is the vanishing cycle, we have $\mu_1\in\delta^\perp\setminus(\delta)$ and $\mu_2\in (\delta,\mu_1)^\perp\setminus(\delta,\mu_1)$, which gives the appropriate number.

Over a reducible curve $C=C_i\cup C_{g-i}$, although the expressions are more complex, the situation is simpler.  We begin by noting that everything in $\Delta_{i,g-i,i:g-i}$ must be in $\Delta_{i,g-i,i:g-i}^0$, because the generators have no common support curve.  As for $\Delta_{i,i,i}$ and $\Delta_{g-i,g-i,g-i}$, they will be precisely the fibers of the lower genus maps $\KM{i}^0\to\M_i$ and $\KM{g-i}^0\to \M_{g-i}$.  On $\Delta_{i,i:g-i,i:g-i}$, we can have any nonzero square trivial line bundle on the component $C_i$, and the second generator can have any restriction to $C_{g-i}$, but the restriction to $C_i$ must be orthogonal to the first, and here, we only divide by two choices in $\Delta_{i:g-i}$ that can be basis elements.  The next component, $\Delta_{g-i,i:g-i,i:g-i}$, can be computed in a similar way.  The final component, $\Delta_{i:g-i,i:g-i,i:g-i}$ starts with an arbitary element of $\Delta_{i:g-i}$, and the second must either have both restrictions orthogonal to those of the first, or else both nonorthogonal, and then we divide by 6 from choices of basis, completing the computation.
\end{proof}

\section{Pluricanonical forms}

In this section, we show that pluricanonical forms on $\overline{\KM{g}}$ extend to any smooth model, allowing us to compute the Kodaira dimension on $\overline{\KM{g}}$ itself, rather than having to work on the set of smooth models.  As such, the goal of this section is to prove

\begin{theorem}
\label{extensiontheorem}
Fix $g\geq 4$ and $i\in\{0,1\}$, and let $\widehat{\ZZM{g}}^i\to\overline{\ZZM{g}}^i$ be any resolution of the singularities.  Then every pluricanonical form defined on $\overline{\ZZM{g}}^{i,reg}$, the smooth locus, extends holomorphically to $\widehat{\ZZM{g}}^i$.  Specifically, for all integers $\ell\geq 0$, we have isomorphisms \[H^0(\overline{\ZZM{g}}^{i,reg},K_{\overline{\ZZM{g}}^{i,reg}}^{\otimes\ell})\cong H^0(\widehat{\ZZM{g}}^i,K_{\widehat{\ZZM{g}}^i}^{\otimes\ell})\]
\end{theorem}

Analogues of this theorem are known for all of the relevant related moduli spaces: $\overline{\M_g}$ is proved in \cite[Theorem 1]{MR664324}, $\overline{\R_g}$ is proved in \cite[Theorem 6.1]{MR2639318}, and the moduli of spin curves in \cite[Theorem 4.1]{MR2551759}.  Our proof will very closely follow the one in \cite{MR2639318} for $\overline{\R_g}$, which can be expected as $\overline{\ZZM{g}}$ is two of the components of $\overline{\R_g}\times_{\overline{\M}_g}\overline{\R_g}$.

Before we can proceed, we need to make a few remarks about the versal deformations of an object $X=(X_1,X_2,\eta_1,\eta_2,\beta_1,\beta_2)\in\overline{\ZZM{g}}$.  Let $\CC_t^{3g-3}$ be the versal deformation space of $Z$, the stabilization of $X_i$ and $\CC_\alpha^{3g-3}$ the versal deformation space of $X$.  There are compatible decompositions 
\begin{eqnarray*}
\CC_\alpha^{3g-3}&\cong&\bigoplus_{p_i\in\Delta_{X_1}^c\cap\Delta_{X_2}}\CC_{\tau_i}\oplus\bigoplus_{p_i\in\Delta_{X_2}^c\cap\Delta_{X_1}}\CC_{\tau_i}\oplus\bigoplus_{p_i\in\Delta_{X_1}^c\cap\Delta_{X_2}^c}\CC_{\tau_i}\\
&&\oplus\bigoplus_{p_i\in\Delta_{X_1}\cap\Delta_{X_2}}\CC_{\tau_i}\oplus\bigoplus_{C_j\subset C}H^1(C_j^\nu,T_{C_j^\nu}(-D_j))\\
\CC_t^{3g-3}&\cong&\bigoplus_{p_i\in\Sing(C)} \CC_{t_i}\oplus\bigoplus_{C_j\subset C}H^1(C_j^\nu,T_{C_j^\nu}(-D_j))
\end{eqnarray*}
where $D_j$ is the sum of the preimages of the nodes under the normalization map.  There is a natural map from the versal deformation space of a $\ZZ_2^2$ curve to that of the underlying stable curve, given by $t_i=\alpha_i^2$ if $t_i=0$ is the locus where the exceptional node $p_i\in\Delta_{X_1}^c\cup \Delta_{X_2}^c$ persists and $t_i=\alpha_i$ otherwise.  Similarly to the discussion in Section 1.2 of \cite{MR2117416}, we can blow up along all of the exceptional components and extend $\eta_1,\eta_2$ using only those in $\Delta_{X_1}^c$ and $\Delta_{X_2}^c$ respectively.

This description makes the rest of the work in Section 6 of \cite{MR2639318} relatively straightforward to generalize.  Set $X_\Delta$ to be the quasi-stable curve with exceptional nodes $\Delta_{X_1}^c\cup\Delta_{X_2}^c$.

\begin{definition}[Elliptic tail]
Let $X$ be a quasi-stable curve, a component $C_j$ is an elliptic tail if it has arithmetic genus 1 and intersects the rest of the curve in a single point.  That point is called an elliptic tail node, and any automorphism of $X$ that is the identity away from $C_j$ is an elliptic tail automorphism.
\end{definition}

\begin{proposition}
Let $\sigma\in \Aut(X)$ be an automorphism in genus $g\geq 4$.  Then $\sigma$ acts on $\CC_\alpha^{3g-3}$ as a quasi-reflection if and only if $X_\Delta$ has an elliptic tail $C_j$ such that $\sigma$ is the elliptic tail involution with respect to $C_j$.
\end{proposition}

The proof of this proposition follows from the proof of \cite[6.6]{MR2639318}.  It implies that the smooth locus of $\overline{\ZZM{g}}$ is the locus where the automorphism group is generated by elliptic tail involutions.  Now that we have determined the smooth locus, we must determine the non-canonical locus.  If $G$ acts on a vector space $V$ by quasi-reflections, then $V/G\cong V$, so we let $H\subset\Aut(X_1,X_2,\eta_1,\eta_2,\beta_1,\beta_2)$ be generated by automorphisms acting as quasi-reflections, that is elliptic tail involutions. Then $\CC^{3g-3}_\alpha/H\cong \CC^{3g-3}_\nu$ where $\nu_i=\alpha_i^2$ if $p_i$ is an elliptic tail node and $\nu_i=\alpha_i$ else.  On $\CC^{3g-3}_\nu$, the automorphisms act without quasi-reflections, so the Reid--Shepherd-Barron--Tai criterion can be applied.

\begin{theorem}[Reid--Shepherd-Barron--Tai Criterion \cite{MR605348,MR763023}]
Let $V$ be a vector space of dimension $d$, $G\subset \GL(V)$ a finite group and $V_0\subset V$ the open set where $G$ acts freely.  Fix $g\in G$, and let $g$ be conjugate to a diagonal matrix with $\zeta^{a_i}$ for $i=1,\ldots,d$ on the diagonal for $\zeta$ a fixed $m^{th}$ root of unity and $0\leq a_i<m$.  If for all $g$ and $\zeta$, we have $\frac{1}{m}\sum_{i=1}^d a_i\geq 1$, then any $n$-canonical form on $V_0/G$ extends holomorphically to a resolution $\widetilde{V/G}$.
\end{theorem}

It is straightforward to check that for $g\geq 4$, we have a noncanonical singularity if $X_\Delta$ has an elliptic tail $C_j$ with $j$-invariant $0$ such that $\eta_1,\eta_2$ are both trivial on $C_j$.  This goes as in \cite{MR2639318}, where the action of $\sigma$ is determined to be as the square of a sixth root of unity in two coordinates for an automorphism of order $6$ and as a cube root of unity in those two coordinates for an order $3$ element.  Both of these fail the Reid--Shepherd-Barron--Tai criterion.

Now, assuming that we have a noncanonical singularity, then we have an automorphism $\sigma$ of order $n$ failing Reid--Shepherd-Barron--Tai.  Our goal is to classify such things, and eventually show that only the examples above exist.  Let $p_{i_0}$, $p_{i_1}=\sigma(p_{i_0})$,$\ldots$, $\sigma^{m-1}(p_{i_0})=p_{i_{m-1}}$ be distinct nodes of the stabilization, $C$, which are permuted by $\sigma$ and not elliptic tail nodes.  The action on the subspace corresponding to these nodes is then given by a matrix \[\left(\begin{array}{cccc}0&c_1&&\\\vdots&&\ddots&\\0&&&c_{m-1}\\c_m&0&\ldots&0\end{array}\right)\] for some complex numbers $c_j$.  We call the pair $(X,\sigma)$ \emph{singularity reduced} if $\prod_{j=1}^m c_j$ is not $1$.

By \cite{MR664324} and \cite[Proposition 3.6]{MR2551759}, we know that there is a deformation $X'$ of $X$ such that $\sigma$ deforms to $\sigma'$, an automorphism of $X'$ such that every cycle of nodes with $\prod_{j=1}^m c_j=1$ is smoothed and the action of $\sigma$ and $\sigma'$ on $\CC_\nu^{3g-3}$ and $\CC_{\nu'}^{3g-3}$ have the same eigenvalues. In particular, one will satisfy Reid--Shepherd-Barron--Tai if and only if the other does.

Now, we fix a pair $(X,\sigma)$ that is singularity reduced and fails the Reid--Shepherd-Barron--Tai inequality.  On $C$, the stabilization, the induced automorphism $\sigma_C$ must either fix all of the nodes or else exchange a single pair of them.  We look at what the action does on the components.  In \cite[Proposition 6.9]{MR2639318} the proof of \cite[Proposition 3.8]{MR2551759} is adapted to the situation of $\overline{\ZM{g}}$, and this proof goes through verbatum, telling us that the action fixes each component of the stable model.  Now, we recall that

\begin{theorem}[{\cite[Page 36]{MR664324}}]
Assume that $(X,\sigma)$ is singularity reduced and fails the Reid--Shepherd-Barron--Tai inequality.  Denote by $\varphi_j$ the induced automorphism on the normalization $C_j^\nu$ of the irreducible component $C_j$ of the stabilization $C$ of $X$.  Then the pair $(C_j^\nu,\varphi_j)$ is one of the following:
\begin{enumerate}
 \item $C_j^\nu$ rational,											and the order of $\varphi_j$ is 2 or 4,
 \item $C_j^\nu$ elliptic,											and the order of $\varphi_j$ is 2,4,3 or 6,
 \item $C_j^\nu$ hyperelliptic of genus 2,			and $\varphi_j$ is the hyperelliptic involution,
 \item $C_j^\nu$ bielliptic of genus 2,					and $\varphi_j$ is the associated involution,
 \item $C_j^\nu$ hyperelliptic of genus 3,			and $\varphi_j$ is the hyperelliptic involution, and
 \item $C_j^\nu$ arbitary,											and $\varphi_j$ is the identity.
\end{enumerate}
\end{theorem}

As pointed out in \cite[Proposition 3.10]{MR2551759}, this rules out the possibility of nodes being exchanged, so the automorphism must fix all nodes and all components on the stable curve.

\begin{proposition}[{\cite[Proposition 6.12]{MR2639318}}]
In the same situation as above, set $D_j$ to be the divisor of the marked points on $C_j^\nu$ that are preimages of nodes.  Then the triples $(C_j^\nu,D_j,\varphi_j)$ are one of the following types, and the contribution to the left hand side of the Reid--Shepherd-Barron--Tai inequality are at least $w_j$:
\begin{enumerate}
 \item $C_j^\nu$ arbitary,											$\varphi_j$ is the identity, 											and $w_j=0$,
 \item Elliptic tails: $C_j^\nu$ is elliptic, $D=p_1^+$ which is fixed by $\varphi_j$, $\varphi_j$ has order 2,3,4 or 6, and $w_j$ is, respectively, $0$, $\frac{1}{3}$, $\frac{1}{2}$ and $\frac{1}{3}$.
 \item Elliptic ladder: $C_j^\nu$ is elliptic and $D=p_1^++p_2^+$, with both points fixed, the automorphism is of order $2$, $3$, or $4$ and $w_j$ is, respectively, $\frac{1}{2}$, $\frac{2}{3}$, and $\frac{3}{4}$
 \item Hyperelliptic tail: $C_j^\nu$ has genus 2, $\varphi_j$ is the hyperelliptic involution, and $D_j=p_1^+$ fixed by $\varphi_j$.  Then $w_j=\frac{1}{2}$.
\end{enumerate}
\end{proposition}

With a bit of case by case work, essentially \cite{MR2639318} Propositions 6.13, 6.14, 6.15 and 6.16, we can see that hyperelliptic tails, elliptic ladders, and elliptic tails of order 4 do not occur, and that there must, in fact, be at least one elliptic tail of order 3 or 6, giving us our restrictions on the curve.  Now, we look to the line bundles.  Because the automorphism must pull back the line bundle to itself on the elliptic curve, it must be trivial on the elliptic tail, and this must hold for both of the Prym line bundles.  Thus, if we start with $(X,\sigma)$ failing Reid--Shepherd-Barron--Tai, then we can deform to a singularity reduced pair $(X',\sigma')$ such that the Reid--Shepherd-Barron--Tai value is constant.  The pair $(X',\sigma')$ must have an elliptic tail with $j$ invariant $0$, the automorphism must be of order 3 or 6, and $\eta_1,\eta_2$ must both be trivial along it.  Thus:

\begin{proposition}
Fix $g\geq 4$.  A point $(X_1,X_2,\eta_1,\eta_2,\beta_1,\beta_2)\in\overline{\ZZM{g}}$ is a non-canonical singularity if and only if $X_\Delta$ has an elliptic tail $C_j$ with $j$-invariant $0$ and $\eta_1|_{C_j}\cong \eta_2|_{C_j}\cong \sh{O}_{C_j}$.
\end{proposition}

\begin{proof}[Proof of Theorem \ref{extensiontheorem}]
Let $\omega$ be a pluricanonical form on $\overline{\ZZM{g}}^{i,reg}$.  We want to show that it lifts to a desingularization of some neighborhood of any point $(X_1,X_2,\eta_1,\eta_2,\beta_1,\beta_2)\in \overline{\ZZM{g}}^i$.  Because this can be done for canonical singularities, we assume that $(X_1,X_2,\eta_1,\eta_2,\beta_1,\beta_2)$ is a general non-canonical singularity, and thus $X_\Delta=C_1\cup_p C_2$ where $(C_1,p)\in \M_{g-1,1}$ and $(C_2,p)\in \M_{1,1}$ with $j(C_2)=0$.  We also assume that $\eta_1|_{C_2}\cong \eta_2|_{C_2}\cong \sh{O}_{C_2}$ and $\eta_i|_{C_1}$ are two arbitrary line bundles on $C_1$, so that we are on a hypersurface in $\Delta_{g-1,g-1}$.  We consider the pencil $\phi:\overline{\M_{1,1}}\to \overline{\ZZM{g}}^i$ given by $\phi(C,p)=C_1\cup_p C$ and line bundles $\eta'_i$ trivial on $C$ and isomorphic to $\eta_i|_{C_i}$ on $C_i$.  As $\phi(\overline{\M_{1,1}})$ does not intersect the ramification locus, then just as in \cite{MR664324} pages 41-44, we can construct an open neighborhood of the pencil, $S$, such that the restriction of $\overline{\ZZM{g}}^i\to\overline{\M_g}$ to $S$ is an isomorphism and every pluricanonical form on the smooth locus extends to a resolution $\hat{S}$ of $S$.  For the arbitrary case, with more than one node, $\omega$ will extend locally to a desingularization, just as in \cite[Theorem 4.1]{MR2551759}.
\end{proof}

Then, Theorem \ref{extensiontheorem} in fact implies the same result for $\overline{\KM{g}}^i$.  This is because $\overline{\ZZM{g}}^i\to\overline{\KM{g}}^i$ is a quotient by $\PSL_2(\FF_2)$.  The action is free except for along $\Delta_{I,III}\cup\Delta_{III,I}\cup\Delta_{III,III}^{\diag}$, where the stabilizer of a point is $\ZZ/2\ZZ$.  Looking at the Reid--Shepherd-Barron--Tai criterion for $m=2$, we find that either the pluricanonical forms extend or we have a quasi-reflection, in which case the pluricanonical forms will also extend.  So, either way, we can see that what we get are the invariants: $H^0(\overline{\KM{g}}^{i,reg},K^{\otimes\ell})\cong H^0(\overline{\ZZM{g}}^i,K^{\otimes \ell})^{\PSL_2(\FF_2)}$, and so, because we can also do this for partial resolutions of $\overline{\ZZM{g}}^i$, we can do this for any resolution $\widehat{\KM{g}}^i$.

We conclude with a statement about the birational geometry of these moduli spaces, justified by the above

\begin{theorem}
\label{maintheorem}
For any $g$, $\overline{\KM{g}}^i$ has general type if there exists a single effective divisor $D\equiv a\lambda-\sum_T b_{\Delta_T}\Delta_T$ where $T$ runs over all boundary components, such that all the ratios $\frac{a}{b_T}$ are less than $\frac{13}{2}$ and the ratios $\frac{a}{b_{II,III,III}}$, $\frac{a}{b_{1,g-1,1:g-1}}$, $\frac{a}{b_{1,1,1}}$, $\frac{a}{b_{g-1,g-1,g-1}}$, $\frac{a}{b_{1,1:g-1,1:g-1}}$, $\frac{a}{b_{g-1,1:g-1,1:g-1}}$, and $\frac{a}{b_{1:g-1,1:g-1,1:g-1}}$ are less than $\frac{13}{3}$.
\end{theorem}

This allows us to begin computing the classes of divisors on the Klein moduli space to determine its Kodaira dimension, and thus begin the study of the birational geometry of these spaces.

\bibliographystyle{alpha}
\bibliography{Klein4Moduli}

\begin{thebibliography}{BCF04}

\bibitem[BCF04]{MR2117416}
Edoardo Ballico, Cinzia Casagrande, and Claudio Fontanari.
\newblock Moduli of {P}rym curves.
\newblock {\em Doc. Math.}, 9:265--281, 2004.

\bibitem[CC03]{MR2007379}
Lucia Caporaso and Cinzia Casagrande.
\newblock Combinatorial properties of stable spin curves.
\newblock {\em Comm. Algebra}, 31(8):3653--3672, 2003.
\newblock Special issue in honor of Steven L. Kleiman.

\bibitem[CLP12]{1206.5498}
Fabrizio Catanese, Michael Loenne, and Fabio Perroni.
\newblock The irreducible components of the moduli space of dihedral covers of
  algebraic curves, 2012.

\bibitem[Cor89]{MR1082361}
Maurizio Cornalba.
\newblock Moduli of curves and theta-characteristics.
\newblock In {\em Lectures on {R}iemann surfaces ({T}rieste, 1987)}, pages
  560--589. World Sci. Publ., Teaneck, NJ, 1989.

\bibitem[Don87]{MR903385}
Ron Donagi.
\newblock Big {S}chottky.
\newblock {\em Invent. Math.}, 89(3):569--599, 1987.

\bibitem[Far12]{MR2976944}
Gavril Farkas.
\newblock Prym varieties and their moduli.
\newblock In {\em Contributions to algebraic geometry}, EMS Ser. Congr. Rep.,
  pages 215--255. Eur. Math. Soc., Z\"urich, 2012.

\bibitem[FL10]{MR2639318}
Gavril Farkas and Katharina Ludwig.
\newblock The {K}odaira dimension of the moduli space of {P}rym varieties.
\newblock {\em J. Eur. Math. Soc. (JEMS)}, 12(3):755--795, 2010.

\bibitem[HM82]{MR664324}
Joe Harris and David Mumford.
\newblock On the {K}odaira dimension of the moduli space of curves.
\newblock {\em Invent. Math.}, 67(1):23--88, 1982.
\newblock With an appendix by William Fulton.

\bibitem[Lud10]{MR2551759}
Katharina Ludwig.
\newblock On the geometry of the moduli space of spin curves.
\newblock {\em J. Algebraic Geom.}, 19(1):133--171, 2010.

\bibitem[Rei80]{MR605348}
Miles Reid.
\newblock Canonical {$3$}-folds.
\newblock In {\em Journ\'ees de {G}\'eometrie {A}lg\'ebrique d'{A}ngers,
  {J}uillet 1979/{A}lgebraic {G}eometry, {A}ngers, 1979}, pages 273--310.
  Sijthoff \& Noordhoff, Alphen aan den Rijn, 1980.

\bibitem[Sie13]{1302.5946}
Charles Siegel.
\newblock The schottky problem in genus five, 2013.

\bibitem[Tai84]{MR763023}
Yung-Sheng Tai.
\newblock Pluricanonical differentials of {H}ilbert modular varieties.
\newblock In {\em Automorphic forms of several variables ({K}atata, 1983)},
  volume~46 of {\em Progr. Math.}, pages 370--377. Birkh\"auser Boston, Boston,
  MA, 1984.

\end{thebibliography}
\end{document}